\documentclass[12pt]{amsart}
\usepackage[utf8]{inputenc}
\usepackage[english]{babel}
\usepackage{amsmath}
\usepackage{amsfonts}
\usepackage{amssymb}
\usepackage{enumerate}
\usepackage[normalem]{ulem}
\usepackage[hidelinks]{hyperref}
\usepackage{xcolor}

\usepackage{amsthm}
\newtheorem{theorem}{Theorem}[section]
\newtheorem{lemma}[theorem]{Lemma}

\newtheorem{question}[theorem]{Question}

\newtheorem{proposition}[theorem]{Proposition}

\theoremstyle{definition}
\newtheorem{definition}[theorem]{Definition}

\newcommand{\I}{\textbf{I}}

\begin{document}
	

	\title{Equivariant homotopy dense subsets in the realm of uniform $G$-ANR spaces}
	
		\author{Sergey A. Antonyan and  Luis A. Mart\'inez-Sánchez
		\\[3pt] }
	
	\address{Departamento de  Matem\'aticas,
		Facultad de Ciencias, Universidad Nacional Aut\'onoma de M\'exico,
		04510, Mexico City,  Mexico}
	\email{antonyan@unam.mx}
	
	
	\address{Departamento de  Matem\'aticas,
		Facultad de Ciencias, Universidad Nacional Aut\'onoma de M\'exico,
		04510, Mexico City,  Mexico}
	\email{luchomartinez9816@hotmail.com}

	\thanks{The first author was supported  by grant IN-107426 from PAPIIT (UNAM) and the second author was supported by grant 829945 from SECIHTI (M\'exico)}
	\begin{abstract} 
	
		Let $G$ be a compact group. The existence of certain $G$-homotopy dense subsets in a metrizable $G$-space $X$ plays a fundamental role, as it is equivalent to $X$ being a $G$-ANR. From this perspective, the present paper develops several applications of this class of $G$-subsets.
		In particular, we prove that for a compact $G$-space $X$ and a metric space $Y$, the mapping space $C(X,Y)$ is a $G$-UA(N)R if and only if $Y$ is a UA(N)R in the sense of Michael. This result is significant because it enables the construction of examples of Lawson metric $G$-semilattices for which the property of being a $G$-UANR is equivalent to uniform local path-connectedness. Moreover, we show that this equivalence holds for every Lawson metric $G$-semilattice whenever $G$ is finite.
		Finally, we analyze the behavior of $G$-homotopy dense subsets when the ambient space is a $G$-A(N)R, thereby introducing the notion of a $G$-A(N)R-pair.
	\end{abstract}
	
	\maketitle\markboth{Sergey A. Antonyan and Luis A. Martínez-Sánchez}{Equivariant homotopy dense subsets}
	
	\section {Introduction}
	
	For a compact group $G$, a recent characterization of $G$-ANRs in terms of equivariant homotopy dense subsets was established in \cite{A8}. Specifically, a metrizable $G$-space $X$ is a $G$-ANR if and only if it contains a $G$-homotopy dense subset $A \subset X$ that is  itself  a $G$-ANR.
	
	This characterization highlights the fundamental role played by $G$-homotopy dense subsets in the equivariant theory of retracts. Motivated by this perspective, in the present paper we explore several applications and examples of such $G$-subsets, and we examine their behavior when the ambient space is a $G$-ANR or a $G$-AR.
	
	The applications developed here are inspired by the notion of a uniform absolute (neighborhood) retract, introduced by H. Toru\'nczyk \cite{HT2,HT3} under the name of a {\it regular absolute (neighborhood) retract}. Five years later,  these objects appeared in E. Michael \cite{EM1} under the name  of    \textit{uniform absolute (neighborhood) retract}, hencefoth abbreviated as UAR or UANR.
	
	\medskip
	
In this paper we continue studying the natural equivariant analogues of these objects, $G$-UARs and $G$-UANRs,  which were introduced  in \cite{A10}.	
	
\medskip
	
	The following result, established in \cite[Theorem 3.4]{KSakaiY}, is 
	of a particular relevance to the present work.
	
	\begin{theorem}\label{Sakai2005Teo3.4}
		\rm{Let $(S,d,\cdot)$ be a Lawson metric semilattice. Then the following conditions are equivalent:
			\begin{enumerate}
				\item [(i)] $S$ is a UANR,
				\item [(ii)] $S$ is uniformly locally contractible,
				\item [(iii)] $S$ is uniformly locally path-connected.
		\end{enumerate}}
	\end{theorem}
	
	A natural question to consider is whether the previous theorem remains valid in the context of Lawson metric $G$-semilattices (see Section \ref{sec4}). In Theorem~\ref{SakaiGfinito}, we provide an affirmative answer to this problem in certain cases, for instance, when the acting group $G$ is finite. However, this hypothesis is not necessary, as illustrated by a result from \cite{A10}, where it was shown that if $G$ is a compact group and $X$ is a metric $G$-space, then the equivalences hold for the hyperspace $2^X$ of nonempty compact subsets of $X$, when it is equipped with a suitable action.
	
	To provide further examples of Lawson metric $G$-semilattices for which these equivalences remain valid, we consider Question \ref{preguntaC(X,Y)GUANR} below. 
	
	We will denote by $C(X,Y)$ the space of continuous maps from a compact $G$-space $X$ to a metric space $Y$, endowed with the compact-open topology and equipped with the $G$-action defined by $(g\cdot f)(x)=f(g^{-1}x)$, where $g\in G$, $f\in C(X, Y)$ and $x\in X$.
	
	\begin{question}\label{preguntaC(X,Y)GUANR}
		\rm{Let $G$ be a compact group, $X$ a compact $G$-space, and $Y$ a UAR (respectively, a UANR). Is $C(X,Y)$ a $G$-UAR (respectively, a $G$-UANR)?}
	\end{question}
	
	Using several properties related to $G$-homotopy dense subsets, we give an affirmative answer to Question \ref{preguntaC(X,Y)GUANR} in Theorem \ref{functionC(X,Y)GUANR}. As a consequence, in Proposition \ref{EquivaC(X,S)equivariant} we present a variety of Lawson metric $G$-semilattices for which the equivalences of Theorem \ref{Sakai2005Teo3.4} hold without assuming that $G$ is finite.
	
	In the course of analyzing the behavior of $G$-homotopy dense subsets when the ambient space is a $G$-A(N)R, we introduce in Definition \ref{G-ANR-pair} the notion of a $G$-A(N)R-pair, thus unifying these concepts within a single framework. Several examples of such pairs are presented throughout Section \ref{sec3}.
		
	\section{Basic definitions}\label{sec2}
	
	Throughout this work, all topological spaces are assumed to be Hausdorff and all maps are assumed to be continuous. In addition, the letter $G$ will denote a Hausdorff topological group with identity element $e\in G$.
	
	For basic notions in the theory of $G$-spaces, we refer the reader to the monographs \cite{GB} and \cite{RP}. To make the exposition self-contained, we recall below several more specialized definitions and fundamental facts.
	
	A $G$-space is a topological space $X$ equipped with a fixed continuous action $G \times X\rightarrow X$ of $G$ on $X$. The image of a pair $(g, x) \in G\times X$ under this action will be denoted by $gx$. If $Y$ is another $G$-space, then a continuous map $f:X\rightarrow Y$ is called equivariant (or a $G$-map) if $f(gx) =gf(x)$ for every $x\in X$ and $g\in G$. A closed isometric $G$-embedding is a closed isometric embedding that is also a $G$-map. Similarly we define a $G$-retraction.
	
	Given a subgroup $H$ of $G$ and a subset $A\subset X$, the $H$-saturation of $A$ is defined as the set $H(A)= \{ha \mid h \in H, a \in A\}$. If $H(A)=A$, then $A$ is said to be an $H$-invariant subset, or simply an $H$-subset. In the particular case where $A$ is $G$-invariant, we simply say that $A$ is invariant. When $A=\{x\}$ for some $x\in X$, we write $H(x)$ instead of $H(\{x\})$; in this case, $H(x)$ is called the $H$-orbit of $x$.
	
	The $H$-orbit space $X/H$ is always equipped with the quotient topology induced by the $H$-orbit projection  $\pi: X\rightarrow X/H$, $\pi (x)=H(x)$ for all $x\in X$.
	
	For any $x\in X$, the stabilizer (or the stationary subgroup) of $x$ is defined by  $G_x=\{g\in G \mid  gx =x\}$. The $H$-fixed point set $X^H$ is defined as $\{x \in X \mid H\subset G_x\}$.
	
	If $X$ is a metrizable $G$-space, a compatible metric $d$ on  $X$ is called $G$-invariant (or simply invariant) provided that $d(gx, gy) =d(x, y)$ for all $g\in G$ and $x, y \in X$. By a metric $G$-space we mean a metric space $(X,d)$ equipped with an action of $G$ such that $d$ is invariant.
	
	When $G$ is compact, every metrizable $G$-space admits a compatible $G$-invariant metric (see, e.g., \cite[Proposition 1.1.12]{RP}). Furthermore, if $d$ is an invariant metric on a metrizable $G$-space $X$ and the orbit space $X/G$ is $T_1$, then the formula 
	\begin{center}
		$\widetilde{d}(G(x), G(y)) = \mbox{inf} \{d(\tilde{x},\tilde{y}) \mid \tilde{x} \in G(x), \tilde{y}\in G(y)\}$
	\end{center}
	defines a metric $\widetilde{d}$ on $X/G$ compatible with the quotient topology (see \cite[Proposition 1.1.12]{RP}). 
	
	By a linear $G$-space we mean a real topological vector space $L$ equipped with a continuous linear action of $G$; that is, $ g(\alpha x + \beta y) =\alpha(gx) +\beta(gy)$ for all $g\in G$, $x, y \in L$ and $\alpha, \beta \in \mathbb{R}$.
	
	Let $(Y,d)$ be a metric space and $X$ a topological space. The set $C(X, Y)$ of continuous maps from $X$ to $Y$ is endowed with the compact-open topology, namely the topology generated by all sets of the form
	
	\begin{center}
		$[K,V]=\{f\in C(X,Y)\mid f(K)\subset V\}$,
	\end{center}
	where $K\subset X$ is compact and $V\subset Y$ is open.
	
	In addition, when $X$ is compact, we consider on $C(X,Y)$ the metric $d^*(p,q)=\sup\{d(p(x),q(x)) \mid x\in X\}$. 
	
	If $X$ is a $G$-space, there is a natural action of $G$ on $C(X,Y)$ defined by $(g\cdot f)(x)=f(g^{-1}x)$. It is well known that this action is continuous whenever either $G$ or $X$ is locally compact (\cite[Proposition 4]{A3}).
	
	\medskip
	
	To introduce the notions of a $G$-UA(N)R and a $G$-UA(N)E, we first recall that a map $f:X\rightarrow Y$ from a metric space $(X,d)$ to a metric space $(Y,\rho)$ is called \textit{uniformly continuous at} a subset $A\subset X$ (see \cite{EM1}) if for each $\varepsilon>0$ there exists $\delta>0$ such that $\rho(f(x),f(a))<\varepsilon$ for all $x\in X$ and $a\in A$ with $d(x,a)<\delta$. Moreover, a neighborhood $U$ of $A$ in $X$ is said to be a \textit{uniform neighborhood} if there exists $\delta > 0$ such that $B_d(A,\delta) = \bigcup_{a \in A} B_d(a,\delta) \subset U$, where $B_d(a,\delta)$  denotes the open ball in $X$ centered at $a$ with radius $\delta$. 
		
	A closed $G$-subset $A$ of a metric $G$-space $X$ is called a uniform $G$-retract (respectively, a uniform neighborhood $G$-retract) of $X$ if there exists a $G$-retraction $r: X\rightarrow A$ (respectively, $r: U\rightarrow A$, for some uniform $G$-neighborhood $U$ of $A$ in $X$) that is uniformly continuous at $A$. In the special case $G=\{e\}$, this definition coincides with the notion of a uniform retract (respectively, a uniform neighborhood retract). In \cite{HT1,HT2}, such a retraction $r$ is referred to as a \textit{regular retraction}.
	
	A metric $G$-space $X$ is called a $G$-equivariant uniform absolute retract (respectively, a $G$-equivariant uniform absolute neighborhood retract), if for any closed isometric $G$-embedding $\iota: X\rightarrow Y$ in a metric $G$-space $Y$, the image $\iota(X)$ is a uniform $G$-retract (respectively, a uniform neighborhood $G$-retract) of $Y$. For brevity, we will refer to such a space $X$ as a $G$-UAR (respectively, a $G$-UANR). When $G=\{e\}$, these notions reduce to the classical concepts of UARs and UANRs introduced by Michael in \cite{EM1}.
	
	A metric $G$-space $Y$ is called a $G$-equivariant uniform absolute extensor (respectively, a $G$-equivariant uniform absolute neighborhood extensor) if, whenever $A$ is a closed $G$-subset of a metric $G$-space $X$ and $f:A \to Y$ is a uniformly continuous $G$-map, there exists a $G$-map $F : X \to Y$ (respectively, $F : U \to Y$, for some uniform $G$-neighborhood $U$ of $A$ in $X$) extending $f$ and uniformly continuous at $A$. In this case, $Y$ is called a $G$-UAE (respectively, a $G$-UANE).
	
	By omitting the word uniform from the definition of a $G$-UA(N)R, we obtain the notion of a $G$-A(N)R. Similarly, the notion of a $G$-A(N)E arises from that of a $G$-UA(N)E by suppressing all references to the metric structure. Note that $G$-A(N)Rs and $G$-A(N)Es are equivariant counterparts of the classical A(N)Rs and A(N)Es, respectively.
	
	When $G$ is compact, a metrizable $G$-space $Y$ is a $G$-A(N)R (respectively, a $G$-UA(N)R) if and only if $Y$ is a $G$-A(N)E (respectively, a $G$-UA(N)E) (see \cite{A3} and \cite{A10}). Throughout this paper, this fact will be used without further explicit mention.
	
	For additional background on the equivariant theory of retracts, we refer the reader to \cite{A2} and \cite{A5}.
		
	Let $X$ and $Y$ be $G$-spaces and let $p,q : X \to Y$ be $G$-maps. A homotopy $F : X \times I \to Y$ from $p$ to $q$ is called a $G$-homotopy if it is a $G$-map, where $X \times I$ is equipped with the diagonal action $g \cdot (x,t) = (gx,t)$. In this case, $p$ and $q$ are said to be $G$-homotopic, and we write $p \simeq_G q$.
	
	Throughout this work, for each $t \in I$ we denote by $F_t : X \rightarrow Y$ the map defined by $F_t(x) = F(x, t)$.
	
	A $G$-map $r:X\to Y$ is a $G$-homotopy equivalence if there exists a $G$-map $s: Y\to X$ such that $s\circ r\simeq_G\text{id}_X$ and $r\circ s\simeq_G\text{id}_Y$. In this case, we say that $X$ and $Y$ are $G$-homotopy equivalent. Here $\text{id}_X$ denotes the identity map on $X$.
	
	If $X\subset Y$ and there exists a $G$-homotopy $F: X\times I\to Y$ such that $F_0$ is the inclusion map $X\hookrightarrow Y$ and $F_1$ is a constant map, then $X$ is said to be $G$-contractible in $Y$. In particular, if $Y=X$, then $X$ is called $G$-contractible.
	
	A metric $G$-space $(X,d)$ is called $G$-uniformly locally contractible (abbreviated $G$-ULC) if for each $\varepsilon>0$, there exists $0<\delta<\varepsilon$ such that the ball $B_d(x,\delta)$ is $G_x$-contractible in $B_d(x,\varepsilon)$ for every $x\in X$. If, in addition, $X$ is a $G$-ANR, then $X$ is said to be a weak $G$-UANR. When $G = \{e\}$, we recover the notion of a weak UANR (see \cite{EM1}).

	In what follows, we will need the following result:
	
	\begin{theorem}[\cite{A10}]\label{lemautil} \rm{
		Let $G$ be a compact group. Then:
		
		\begin{enumerate}
			\item[(i)] Every $G$-UA(N)R is both a $G$-A(N)R and a UA(N)R.
			\item[(ii)] Every $G$-UANR is $G$-ULC and hence a weak $G$-UANR.
		\end{enumerate}}
	\end{theorem}

	\medskip

	\section{\texorpdfstring{$G$-}-ANR-pairs}\label{sec3}
	Recall that a $G$-subset $Z$ of a $G$-space $X$ is called $G$-homotopy dense in $X$ if there exists a $G$-homotopy $F:X\times I\rightarrow X$ such that $F_0=\mbox{id}_X$ and $F(X\times (0,1])\subset Z$ (see \cite{A8}). 
	
	For example, consider the Euclidean space $\mathbb{R}^n$ endowed with the canonical action of the orthogonal group $O(n)$, and let $\mathbb{B}^n$ denote the closed unit ball. We show that $\operatorname{int}(\mathbb{B}^n)$ is $O(n)$-homotopy dense in $\mathbb{B}^n$. To this end, define a map $F\colon \mathbb{B}^n\times I\to \mathbb{B}^n$ by
	\begin{center}
		$F(x,t)=(1-t)x,\quad$ $x\in \mathbb{B}^n,\quad$ $t\in I$.
	\end{center}
	
	Clearly, $F_0=\operatorname{id}_{\mathbb{B}^n}$ and $F(\mathbb{B}^n\times (0,1])\subset \operatorname{int}(\mathbb{B}^n)$. Therefore, $\operatorname{int}(\mathbb{B}^n)$ is $O(n)$-homotopy dense in $\mathbb{B}^n$.
	
	Now consider the Hilbert cube $Q=[-1,1]^{\infty}$ equipped with the canonical $\mathbb{Z}_2$-action. By Proposition~\ref{producthomotopydense}, it follows that $(-1,1)^{\infty}$ is $\mathbb{Z}_2$-homotopy dense in $Q$.

	\begin{proposition}\label{producthomotopydense}
		\rm{Let $\{X_j\mid j\in J\}$ be a collection of $G$-spaces, and for each $j\in J$, let $A_j$ be a $G$-homotopy dense subset of $X_j$. Then $A=\prod_{j\in J}A_j$ is $G$-homotopy dense in the product $G$-space $X=\prod_{j\in J}X_j$.}
	\end{proposition}
	
	\begin{proof}
		For each $j\in J$, let $F_j\colon X_j\times I\to X_j$ be a $G$-homotopy such that $F_j(X_j\times (0,1])\subset A_j$ and $F_j(x,0)=x$ for all $x\in X_j$.  Let $\pi_j\colon X\to X_j$ denote the projection map.
		
		Define a $G$-map $F\colon X\times I\to X$ by
		\begin{center}
			$F(x,t)=\big(F_j(\pi_j(x),t)\big)_{j\in J},\quad x\in X,\quad t\in I$.
		\end{center}  
		
		Clearly, $F_0=\text{id}_X$ and $F(X\times (0,1])\subset A$.  Hence, $A$ is $G$-homotopy dense in $X$, as desired.    
	\end{proof}
	
	Now we introduce the notion of a $G$-ANR-pair, which will allow us to present further examples of $G$-homotopy dense subsets.
	
	\begin{definition}\label{G-ANR-pair}
		Let $X$ be a metrizable $G$-space and $A$ a $G$-subset of $X$. The pair $(X,A)$ is called a $G$-ANR-pair, if for every metrizable $G$-space $Z$ and every closed $G$-subset $B \subset Z$, any $G$-map $f : B \to X$ extends to a $G$-map $F : U \to X$ for some $G$-neighborhood $U$ of $B$ in $Z$, with the additional property that $F(U \setminus B) \subset A$. If one can take $U = Z$, then $(X,A)$ is called a $G$-AR-pair.
	\end{definition}
	
	Observe that, when $G$ is the trivial group, Definition \ref{G-ANR-pair} reduces to the classical notion of an A(N)R-pair (see, e.g., \cite[p. 266]{VM2}). Moreover, it follows immediately that if $(X,A)$ is a $G$-AR-pair (respectively, a $G$-ANR-pair), then both $X$ and $A$ are $G$-ARs (respectively, $G$-ANRs).
	
	In Theorem \ref{G-pair and G-homotopy dense}, we establish a connection between $G$-ANR-pairs and $G$-homotopy dense subsets. 
	
	

	\begin{theorem}\label{G-pair and G-homotopy dense}
		\rm{Let $G$ be a compact group. Then $(X,A)$ is a $G$-AR-pair (respectively, a $G$-ANR-pair) if and only if $X$ is a $G$-AR (respectively, a $G$-ANR) and $A$ is $G$-homotopy dense in $X$.}
	\end{theorem}
	
	\begin{proof} We prove the $G$-ANR case.
		
		Suppose that $(X,A)$ is a $G$-ANR-pair. Since  $X$ is a $G$-ANR, it remains to show that $A$ is $G$-homotopy dense in $X$.  Define a $G$-map $f: X\times \{0\}\rightarrow X$  by $f(x,0)=x$. By our hypothesis, there exist a $G$-neighborhood $V$ of $X\times \{0\}$ in $X\times I$, and a $G$-extension $p: V\rightarrow X$ such that $p\big(V\setminus (X\times \{0\})\big)\subset A$.
		
		Let $\sigma\leq 1$ be an admissible invariant metric on $X$. Then the function 
		
		\begin{center}
			$d:(X\times I)^2\rightarrow [0,\infty),\quad$ $d((x,t),(z,s))=\sigma(x,z)+|t-s|$
		\end{center}
		is an admissible invariant metric on $X\times I$.
		
		Next, consider the map $\lambda: X\rightarrow (0,1]$ defined by
		\begin{center}
			$\lambda(x)=\frac{1}{2}d((x,0), (X\times I)\setminus V),\quad x\in X$.
		\end{center} 
		
		Clearly, $\lambda$ is invariant	 and satisfies
		\begin{center}
			$\{(x,t)\in X\times I\mid t\leq \lambda(x)\}\subset V.$
		\end{center}
		
		We use the map $\lambda$ to define a $G$-map $K:X\times I\rightarrow X$ by 
		\begin{center}
			$K(x,t)=p(x,\lambda(x)t),\quad x\in X,\quad t\in I$.
		\end{center} 
		
		Observe that $K_0=\text{id}_X$ and $K(X\times (0,1])\subset A$, which shows that $A$ is $G$-homotopy dense in $X$.
		
		Conversely, suppose that $X$ is a $G$-ANR and that $A$ is $G$-homotopy dense in $X$. Then there is a $G$-homotopy $F: X\times I\rightarrow X$ such that $F_0=\mbox{id}_X$ and $F(X\times (0,1])\subset A$.
		
		Let $B$ be a closed $G$-subset of a metrizable $G$-space $Y$, and let $r : B \rightarrow X$ be a $G$-map. Since $X$ is a $G$-ANR, there exist a $G$-neighborhood $U$ of $B$ in $Y$ and a $G$-extension $s : U \to X$ of $r$.
		
		By the compactness of $G$, there is an invariant metric $\sigma$ on $Y$ such that $\sigma\leq 1$. Define a map $R: U\rightarrow X$ by $R(u)=F\big(s(u),\sigma(u, B)\big)$.
		
		Observe that $R$ is a $G$-extension of $r$. Moreover, if $u\in U\setminus B$, we have $R(u)=F\big(s(u),\sigma(u,B)\big)\in F(X\times (0,1])\subset A$. Therefore, $R(U \setminus B) \subset A$, which completes the proof.
	\end{proof}
	
	It follows from the examples given at the beginning of this section that $\big(\mathbb{B}^n, \operatorname{int}(\mathbb{B}^n)\big)$ is an $O(n)$-AR-pair, and that $\big(Q, (-1,1)^\infty\big)$ is a $\mathbb{Z}_2$-AR-pair.
	
	\begin{proposition}\label{ANR-pair + contrac = AR-pair}
		\rm{Let $G$ be a compact group and $(X,A)$ a $G$-ANR-pair. Then $(X,A)$ is a $G$-AR-pair if and only if $A$ is $G$-contractible.}
	\end{proposition}
	
	\begin{proof}
		
		Necessity is evident, so we only prove sufficiency. Suppose that $A$ is $G$-contractible. 
		
		We first observe that $X$ is $G$-homotopy equivalent to $A$, and hence $X$ is also $G$-contractible. Indeed, let $F: X\times I\rightarrow X$ be a $G$-homotopy such that $F_0=\text{id}_X$ and $F(X\times (0,1])\subset A$. 
		
		Define a $G$-map $p: X\rightarrow A$ by $p(x)=F_1(x)$, and let $q:A\rightarrow X$ denote the inclusion map. Then the restriction $F\restriction_{A\times I}: A\times I\rightarrow A$ is a $G$-homotopy from $\text{id}_A$ to $p\circ q$, and $F$ is a $G$-homotopy from $\text{id}_X$ to $q\circ p$. Therefore, $X$ and $A$ are $G$-homotopy equivalent.
		
		Since both $X$ and $A$ are $G$-contractible $G$-ANRs, it follows from \cite[Theorem 6]{A2} that $X$ and $A$ are $G$-ARs. Moreover, because $A$ is $G$-homotopy dense in $X$, Theorem \ref{G-pair and G-homotopy dense} implies that $(X,A)$ is a $G$-AR-pair. This completes the proof. 
	\end{proof}
	
	Let $G$ be a finite group and $L$ a locally convex linear $G$-space. Given a convex $G$-subset $V$ of $L$, we prove that $(\overline{V}, V)$ is a $G$-AR-pair, and, by Theorem \ref{G-pair and G-homotopy dense}, it follows that $V$ is $G$-homotopy dense in $\overline{V}$.
	
	Since $V$ is a UANR, it follows from \cite[Theorem 2]{KS1} that  $V$ is homotopy dense in $\overline{V}$. Furthermore, $\overline{V}$ is itself  a UANR. Hence, by Theorem \ref{G-pair and G-homotopy dense}, the pair $(\overline{V}, V)$ is an ANR-pair. Moreover, since $V$ is contractible, Proposition \ref{ANR-pair + contrac = AR-pair} yields that $(\overline{V}, V)$ is an AR-pair  (see also \cite[Corollary 4.2.15]{VM2}).
	
	Now, let $X$ be a metrizable $G$-space, $A$ a closed $G$-subset of $X$, and $f: A\rightarrow \overline{V}$  a $G$-map. Since $(\overline{V},V)$ is an AR-pair, there exists a continuous extension $k: X\rightarrow \overline{V}$ of $f$ such that $k(X\setminus A)\subset V$. 
	
	Define a map $F: X\rightarrow \overline{V}$ by 
	
	\begin{center}
		$F(x)=\dfrac{1}{|G|}\displaystyle\sum\limits_{g\in G}g^{-1}k(gx),\quad$ $x\in X$.
	\end{center}
	
	Clearly, $F$ is a $G$-extension of $f$. Moreover, for each $x\in X\setminus A$, the value $F(x)$ is a convex combination of elements of $V$, and hence $F(x)\in V$. Therefore, $F(X\setminus A)$ is contained in $V$, and consequently $(\overline{V},V)$ is a $G$-AR-pair. We thus obtain the following result.
	
	\begin{proposition}\label{ConvexARpairfi}
		\rm{Let $G$ be a finite group and $L$ a locally convex linear $G$-space. If $V$ is a convex $G$-subset of $L$, then $V$ is $G$-homotopy dense in $\overline{V}.$}
	\end{proposition}
	
The following result shows that the finiteness assumption on $G$ in Proposition \ref{ConvexARpairfi} can be dropped whenever $L$ is a finite-dimensional normed linear $G$-space.

In the proof, we shall make use of the relative interior of a convex subset $V \subset L$, namely
	
	\begin{center}
		$\text{ri}(V)=\{x\in V\mid \exists\, \varepsilon>0\, \text{such that}\, B(x,\varepsilon)\cap \text{aff}(V)\subset V\}$,
	\end{center}
	where 
	
	\begin{center}
		$\text{aff}(V)=\left\{\displaystyle\sum_{i=1}^nt_ix_i\mid t_i\in\mathbb{R},\, x_i\in V,\, \sum_{i=1}^nt_i=1, \  n\in\mathbb{N}\right\}$
	\end{center}
	denotes the affine hull of $V$. We write $\operatorname{conv}(V)$ for the convex hull of $V$.
	
	\begin{proposition}
		\rm{Let $G$ be a compact group and let $V$ be a convex $G$-subset of a finite-dimensional normed linear $G$-space $L$. Then $(\overline{V},V)$ is a $G$-AR-pair.}
	\end{proposition}
	
	\begin{proof}
		Since $\overline{V}$ is a $G$-AR, it suffices to show that $V$ is $G$-homotopy dense in $\overline{V}$ and then apply Theorem \ref{G-pair and G-homotopy dense}.
		
		As $L$ is finite dimensional, by \cite[Theorem 2.3.1]{RW}, there exists $w \in \mathrm{ri}(V)$. Since $V$ is invariant, its relative interior $\mathrm{ri}(V)$ is also invariant, and hence the orbit $G(w)$ is contained in $\mathrm{ri}(V)$. Moreover, by \cite[Theorem 2.3.5]{RW}, the set $\mathrm{ri}(V)$ is convex. Consequently, $\operatorname{conv}(G(w)) \subset \mathrm{ri}(V)$.
		
		Consider the map  $f: G\rightarrow \text{conv}(G(w))$ defined by $f(h)=hw$, $h\in G$. Since $L$ is finite-dimensional and $G(w)$ is a compact subset of $L$, it then follows from \cite[Theorem 3.20]{WR} that $\operatorname{conv}(G(w))$ is compact. Therefore, by \cite[Theorem 3.27]{WR}, the Haar integral 
		\[
		v = \int_G f
		\]
		is well defined and 	$v$	belongs to $\operatorname{conv}(G(w)) \subset \mathrm{ri}(V)$.
		
		We now define a map $H: \overline{V}\times I\rightarrow \overline{V}$ by $H(x,t)=(1-t)x+tv$. Since $v$ is a $G$-fixed point, i.e.,     $v \in L^G$, the map $H$ is $G$-equivariant. Clearly, $H_0 = \mathrm{id}_{\overline{V}}$, and by \cite[Theorem 2.3.4]{RW}, we have 
		\[
		H\big(\overline{V} \times (0,1]\big) \subset V.
		\]
		
		Thus, $V$ is $G$-homotopy dense in $\overline{V}$, and the conclusion follows from Theorem \ref{G-pair and G-homotopy dense}.
	\end{proof}
	
	In Theorems \ref{cone(A)GHDcone(X)}, \ref{GxA G-HD GxD}, \ref{condicionesGhomotopydense}, and~\ref{npower}, we study several examples of	$G$-ANR-pairs and their relationship with some classical functorial constructions.
	
	Let $X$ be a $G$-space. The cone of $X$, denoted by $\mathrm{Cone}(X)$, is the quotient set $(I\times X)/(\{0\}\times X)$ with quotient map $\eta:I\times X\rightarrow \mathrm{Cone}(X)$. For each $(t,x)\in I\times X$, we write $tx=\eta(t,x)$, and we denote by $\theta:=0x$ the vertex of $\mathrm{Cone}(X)$.
	
	Throughout, we equip $\mathrm{Cone}(X)$ with the weak topology. In this topology, a subset $U \subset \mathrm{Cone}(X)$ containing $\theta$ is open if and only if $\eta^{-1}(U)$ is open in $I\times X$ and there exists $\varepsilon > 0$ such that $[0,\varepsilon)\times X \subset \eta^{-1}(U)$. If $\theta \notin U$, then $U$ is open precisely when $\eta^{-1}(U)$ is open in $I\times X$. 
	
	Note that the weak topology is strictly coarser than the quotient topology induced by $\eta$, although the two coincide whenever $X$ is compact.
	
	Define a $G$-action on $\mathrm{Cone}(X)$ by
	
	\begin{center}
		$g\cdot tx=t(gx)\quad$ for all $g\in G$ and $tx\in$ $\mathrm{Cone}(X)$.
	\end{center}
	
	\begin{lemma}\label{lemmacone(A)GHDcone(X)}
		\rm{Let $G$ be a topological group, $X$ a $G$-space and $A$ a $G$-subset of $X$. If $A$ is $G$-homotopy dense in $X$, then $\mathrm{Cone}(A)$ is $G$-homotopy dense in $\mathrm{Cone}(X)$.}
	\end{lemma}
	
	\begin{proof}
		Since $A$ is $G$-homotopy dense in $X$, there exists a $G$-homotopy $F: X\times I\rightarrow X
		$ such that $F_0=\mbox{id}_X$ and $F(X\times (0,1])\subset A$. Define a  map $K:\mathrm{Cone}(X)\times I\rightarrow \mathrm{Cone}(X)$ by $K(tx,s)=tF(x,s)$. 
		
		Note that $K$ is well-defined,  $K(\mathrm{Cone}(X)\times (0,1])\subset \mathrm{Cone}(A)$ and that $K_0$ is the identity map of $\mathrm{Cone}(X)$. Moreover, for each $g\in G$, $x\in X$ and $s,t\in I$, we have
		
		\begin{center}
			$K\big(g(tx,s)\big)=K\big(t(gx),s\big)=t\big(gF(x,s)\big)= g\big(tF(x,s)\big)=gK(tx,s)$
		\end{center}
		showing that $K$ is equivariant. Thus, it remains to show that $K$ is continuous.
		 
		Fix $s\in I$ and let us prove that $K$ is continuous at $(\theta,s)$. Suppose that $U$ is a neighborhood of $K(\theta,s)=\theta$. Choose $\varepsilon>0$ such that  $[0,\varepsilon)\times X\subset \eta^{-1}(U)$, where $\eta:I \times X\rightarrow \mathrm{Cone}(X)$ is the quotient map. Since the set $W=\eta([0,\varepsilon)\times X)$ is an open neighborhood of $\theta$, it suffices to show that $K(W\times I)\subset U$.
		
		Let $(w,r)\in W\times I$. We may assume that $w\neq \theta$, so $w=tx$ for some $t\in (0,\varepsilon)$. Then $K(w,r)=K(tx,r)=tF(x,r)\in \eta([0,\varepsilon)\times X)\subset U$, and hence, $K$ is continuous at $(\theta,s)$.
		
		Now let $(az,b)\in \mathrm{Cone}(X)\times I$ with $a>0$, and suppose that $V$ is a neighborhood of $K(az,b)=aF(z,b)$. 
		
		There exist an open subset $U$ of $X$ and $\delta>0$, with $a-\delta>0$, such that $\big(a,F(z,b)\big)\in (a-\delta,a+\delta)\times U\subset \eta^{-1}(V)$. By the continuity of $F$, there exist neighborhoods $Z$ of $z$ in $X$ and $B$ of $b$ in $I$ such that $F(Z\times B)\subset U$.
		
		Define $D=\eta\big((a-\delta,a+\delta)\times Z\big)\times B$. Then $D$ is a neighborhood of $(az,b)$ and $K(D)\subset V$. Thus, $K$ is continuous in $(az,b)$. This finishes the proof.
		
	\end{proof} 
	
	\begin{theorem}\label{cone(A)GHDcone(X)}
		\rm{Let $G$ be a compact group and $(X,A)$ be a $G$-ANR-pair. Then $\big(\mathrm{Cone}(X),\mathrm{Cone}(A)\big)$ is a $G$-AR-pair.}
	\end{theorem}
	
	\begin{proof}
		Since $X$ and $A$ are $G$-ANRs, it follows from \cite[Proposition 2.2]{A9} that both $\mathrm{Cone}(X)$ and $\mathrm{Cone}(A)$ are $G$-ARs. Moreover, by Theorem \ref{G-pair and G-homotopy dense}, $A$ is $G$-homotopy dense in $X$, and by Lemma \ref{lemmacone(A)GHDcone(X)} we have that $\mathrm{Cone}(A)$ is $G$-homotopy dense in $\mathrm{Cone}(X)$. Hence, by Theorem \ref{G-pair and G-homotopy dense}, we conclude that $\big(\mathrm{Cone}(X), \mathrm{Cone}(A)\big)$ is a $G$-AR-pair.
	\end{proof}
	
	Recall from \cite[Ch. II, \S 2]{GB}, that given a subgroup $H$ of $G$ and an $H$-space $X$, the \textit{twisted product} $G\times_H X$ is the $H$-orbit space of the $H$-space $G\times X$, where $H$ acts on $G\times X$ by $h(g,x)=(gh^{-1},hx)$. For each $(g,x)\in G\times X$, the $H$-orbit $H(g,x)$ will be denoted by $[g,x]$, that is, $[g,x]=\{(gh^{-1},hx) \mid h\in H\}$. There exists a canonical action of $G$ on $G\times_HX$ defined by $g'[g,x]=[g'g,x]$. 
		
	\begin{lemma}\label{lemmaGxA G-HD GxD}
		\rm{Let $H$ be a closed subgroup of $G$, and let $X$ be an $H$-space. If $A\subset X$ is $H$-homotopy dense in $X$, then $G\times_H A$ is $G$-homotopy dense in $G\times_H X$.}
	\end{lemma}
	
	\begin{proof}
		Throughout the proof we use the canonical $G$-homeomorphism $\alpha: (G\times_HX)\times I\rightarrow G\times_H(X\times I)$ defined by $\alpha([g,x],t)=[g,(x,t)]$.
		
		By our hypothesis, there exists an $H$-homotopy $F:X\times I\rightarrow X$ such that $F_0=\mbox{id}_X$ and $F(X\times (0,1])\subset A$. Since $F$ is an $H$-map, the map $K: G\times_H (X\times I)\rightarrow G\times_H X$ defined by 
		
		\begin{center}
			$K([g,(x,t)])=[g,F(x,t)],\quad [g,(x,t)]\in G\times_H(X\times I)$,
		\end{center}
		is a $G$-map. Consequently, the composition $J=K\circ \alpha: (G\times_HX)\times I\rightarrow G\times_HX$ is a $G$-map.
		
		We now show that $J$ is the desired $G$-homotopy. 
		Indeed, for each $[g,x] \in G \times_H X$ and $t \in (0,1]$, we have
		
		\begin{center}
			$J([g,x],0)=K([g,(x,0)])=[g,F(x,0)]=[g,x]$,
		\end{center}
		and
		\begin{center}
			$J([g,x],t)=K([g,(x,t)])=[g,F(x,t)]\in G\times_HA$.
		\end{center}
		Thus, $J_0=\mbox{id}_{G\times_HX}$ and $J\big((G\times_HX)\times (0,1]\big)\subset G\times_HA$. Therefore, $G\times_HA$ is $G$-homotopy dense in $G\times_HX$. 
	\end{proof}

	\begin{theorem}\label{GxA G-HD GxD}
		\rm{Let $H$ be a closed subgroup of a compact Lie group $G$. If $(X,A)$ is an $H$-ANR-pair, then $(G\times_HX,G\times_H A)$ is a $G$-ANR-pair.}
	\end{theorem}
	
	\begin{proof}
		Since $X$ and $A$ are $H$-ANRs, it follows from \cite[Proposition 3.1]{A4} that $G \times_H X$ and $G \times_H A$ are $G$-ANRs. Moreover, by Lemma \ref{lemmaGxA G-HD GxD} we have that $G \times_H A$ is $G$-homotopy dense in $G \times_H X$. Therefore, by Theorem \ref{G-pair and G-homotopy dense}, we deduce that $(G \times_H X, G \times_H A)$ is a $G$-ANR-pair.
	\end{proof}
	
	\begin{lemma}\label{lemmacondicionesGhomotopydense}
		\rm{Let $X$ be a $G$-space and $Z$ a $G$-homotopy dense subset of $X$. Then, for each closed subgroup $H$ of $G$, the following conditions are satisfied:
			
			\begin{enumerate}
				\item [(i)] $Z^H$ is homotopy dense in $X^H$.
				\item [(ii)] The $H$-orbit space $Z/H$ is homotopy dense in $X/H$.
				\item [(iii)] $Z$ is $H$-homotopy dense in $X$.
		\end{enumerate}}
	\end{lemma}
	
	\begin{proof}
		There exists a $G$-homotopy $F: X\times I\rightarrow X$ such that $F_0=\mbox{id}_X$ and $F(X\times (0,1])\subset Z$. 
		
		Let $H$ be a closed subgroup of $G$. Then $F$ induces a continuous map $F^H: (X\times I)^H\rightarrow X^H$ defined by $F^H(x,t)=F(x,t)$. Notice that $(X\times I)^H=X^H\times I$, so $F^H$ is a map from $X^H\times I$ to $X^H$. Moreover, $F^H_0=\mbox{id}_{X^H}$ and $F^H(X^H\times (0,1])\subset X^H\cap Z=Z^H$; hence, $Z^H$ is homotopy dense in $X^H$. This completes the proof of (i).
		
		To prove (ii), define a continuous map  $K: (X\times I)/H\rightarrow X/H$ by $K([(x,t)])=[F(x,t)]$ for each $[(x,t)]\in (X\times I)/H$. Since the map $\alpha: (X/H)\times I\rightarrow (X\times I)/H$ given by $\alpha([x],t)=[(x,t)]$ is a homeomorphism, the composition $J=K\circ \alpha: (X/H)\times I\rightarrow X/H$ is continuous. 
		
		For each $[x]\in X/H$ and $t\in (0,1]$, we have
		
		\begin{center}
			$J([x],0)=K([(x,0)])=[F(x,0)]=[x]$
		\end{center}
		and
		\begin{center}
			$J([x],t)=K([(x,t)])=[F(x,t)]\in Z/H$.
		\end{center}
		Thus, $J_0=\mbox{id}_{X/H}$ and $J\big((X/H)\times (0,1]\big)\subset Z/H$. Hence,  $Z/H$ is homotopy dense in $X/H$, which proves (ii).
		
		Finally, since $F$ is an $H$-map, condition (iii) follows directly.
	\end{proof}
	
	Let $G$ be a compact group and suppose that $(X,A)$ is a $G$-ANR-pair. By \cite[Theorem 1.1]{A5}, the orbit spaces $X/G$ and $A/G$ are ANRs. In addition, $A/G$ is homotopy dense in $X/G$. Consequently, by Theorem \ref{G-pair and G-homotopy dense}, we deduce that $(X/G,A/G)$ is an ANR-pair.
	
	Now suppose that $H$ is a closed subgroup of $G$. By Lemma \ref{lemmacondicionesGhomotopydense}, we have that $A^H$ is homotopy dense in $X^H$. Since both $X^H$ and $A^H$ are ANRs, by Theorem \ref{G-pair and G-homotopy dense} we obtain that $(X^H, A^H)$ is an ANR-pair.
	
	Finally, it follows from \cite[Corollary 6.4]{AAM} that $X$ and $A$ are $H$-ANRs. Applying once again Theorem \ref{G-pair and G-homotopy dense}, we conclude that $(X, A)$ is an $H$-ANR-pair.
	
	We summarize the preceding results in the following theorem.
	
	\begin{theorem}\label{condicionesGhomotopydense}
		\rm{Let $H$ be a closed subgroup of a compact group $G$. If $(X,A)$ is a $G$-ANR-pair, then $(X^H,A^H)$ and $(X/G,A/G)$ are ANR-pairs. Furthermore, $(X,A)$ is an $H$-ANR-pair.}
	\end{theorem}
	
	Recall that the $n$-th symmetric power of a topological space $X$ is the $S_n$-orbit space $SP^n(X)=X^n/S_n$, where the symmetric group $S_n$ acts on $X^n$ by permuting coordinates, with $n\geq 2$.
	
	If $X$ is equipped with an action $\theta:G\times X\rightarrow X$ of a topological group $G$, it follows from \cite{HJ1} that $SP^n(X)$ is a $G$-space with the action $\overline{\theta}: G\times SP^n(X)\rightarrow SP^n(X)$ defined by
	\begin{center}
		$\overline{\theta}(g,[x_1,\dots,x_n])=[gx_1,\dots,gx_n],\quad g\in G,\quad [x_1,\dots,x_n]\in SP^n(X)$.
	\end{center}
	
	Moreover, it was shown in \cite[Theorem 4.2]{HJ1} that for any compact group $G$, the $n$-th symmetric power $SP^n(X)$ is a $G$-AR (respectively, a $G$-ANR) whenever $X$ is a $G$-AR (respectively, a $G$-ANR). Hence, we obtain the following result.
	
	\begin{theorem}\label{npower}
		\rm{Let $G$ be a compact group. If $(X,A)$ is a $G$-AR-pair (respectively, a $G$-ANR-pair), then $\big(SP^n(X),SP^n(A)\big)$ is a $G$-AR-pair (respectively, a $G$-ANR-pair).}
	\end{theorem}
	
	\begin{proof}
		If $(X,A)$ is a $G$-ANR-pair, then both $SP^n(X)$ and $SP^n(A)$ are $G$-ANRs. Moreover, there exists a $G$-homotopy $F:X\times I\rightarrow X$ such that $F_0=\mathrm{id}_X$ and $F(X\times(0,1])\subset A$.
		
		Consider the map $K: SP^n(X)\times I\rightarrow SP^n(X)$ defined by
		\begin{center}
			$K([x_1,\dots,x_n],t)=[F(x_1,t),\dots,F(x_n,t)]$
		\end{center}
		for all $[x_1,\dots,x_n]\in SP^n(X)$ and $ t\in I$.
		
		Then $K$ is a $G$-map, $K_0=\mathrm{id}_{SP^n(X)}$ and $K(SP^n(X)\times (0,1])\subset SP^n(A)$. Hence, $SP^n(A)$ is $G$-homotopy dense in $SP^n(X)$, and by Theorem \ref{G-pair and G-homotopy dense}, we conclude that $\big(SP^n(X),SP^n(A)\big)$ is a $G$-ANR-pair.
		
		The case in which $(X,A)$ is a $G$-AR-pair is similar.
	\end{proof}
	
	In Theorem \ref{functionC(X,Y)GUANR} we provide an affirmative answer to Question \ref{preguntaC(X,Y)GUANR}.
	
	\begin{theorem}\label{functionC(X,Y)GUANR}
		\rm{Let $G$ be a compact group, $X$ a compact $G$-space and $(Y,d)$ a metric space. Then $Y$ is a UAR (respectively, a UANR) if and only if $(C(X,Y),d^*)$ is a $G$-UAR (respectively, a $G$-UANR).}
	\end{theorem}
	
	The proof of Theorem \ref{functionC(X,Y)GUANR} is based on the following four lemmas established in \cite{A10}, together with Lemmas \ref{functionhomotopydense} and \ref{ejemplouniformretract}.
	
	\begin{lemma}[\cite{A10}]\label{UAE grupo compacto no Lie}
		\rm{Let $V$ be a closed convex $G$-subset of a Banach $G$-space $L$. If $G$ is compact, then $V$ is a $G$-UAR.} 
	\end{lemma}

	\begin{lemma}[\cite{A10}]\label{retractuniform}
		\rm{Let $(Z,d)$ be a metric $G$-space and $Y$ a closed $G$-subset of $Z$. If $Z$ is a $G$-UAE (respectively, a $G$-UANE) and $Y$ is a uniform $G$-retract (respectively, a uniform neighborhood $G$-retract) of $Z$, then $Y$ is a $G$-UAE (respectively, a $G$-UANE).}
	\end{lemma}
	
	\begin{lemma}[\cite{A10}]\label{G-AR+G-UANR=G-UAR}
		\rm{Let $G$ be a compact group and $(Y,d)$ a metric $G$-space. Then $Y$ is a $G$-UAR if and only if $Y$ is a $G$-AR and a $G$-UANR.}
	\end{lemma}
	
	\begin{lemma}[\cite{A10}]\label{G-UANRpreserveHD}
		\rm{Let $G$ be a compact group, $(Y,d)$ a metric $G$-space and  $Z$  a $G$-homotopy dense subset of $Y$. If $Y$ is a $G$-UANE then $Z$ is a $G$-UANE.}
	\end{lemma}
	
	\begin{lemma}\label{functionhomotopydense}
		\rm{Let $X$ be a compact $G$-space and $Y$ a topological space. If $A\subset Y$ is homotopy dense in $Y$, then $C(X,A)$ is $G$-homotopy dense in $C(X,Y)$.}
	\end{lemma}
	
	\begin{proof}
		There exists a  homotopy $F:Y\times I\rightarrow Y$ such that $F_0=\text{id}_Y$ and $F(Y\times (0,1])\subset A$. Define a map $J:C(X,Y)\times X\times I \rightarrow Y\times I$ by
		
		\begin{center}
			$J(p,x,t)=F\big(p(x),t\big),\quad p\in C(X,Y),\quad x\in X,\quad t\in I$.
		\end{center}
		
		By the compactness of $X$, this map is continuous. Therefore, the map $K:C(X,Y)\times I\rightarrow C(X,Y)$ defined by
		
		\begin{center}
			$K(p,t)(x)=F\big(p(x),t\big),\quad p\in C(X,Y),\quad x\in X,\quad t\in I$
		\end{center}
		is continuous. We now verify that $K$ is the desired $G$-homotopy.
		
		For each $p\in C(X,Y)$ and $x\in X$, we have
		\begin{center}
			$K(p,0)(x)=F\big(p(x),0\big)=p(x)$,
		\end{center}
		then $K(p,0)=p$, and thus $K_0=\text{id}_{C(X,Y)}$. Moreover, if $t>0$, for each $x\in X$ we have $K(p,t)(x)=F\big(p(x),t\big)\in A$, which shows that $K(C(X,Y)\times (0,1])\subset C(X,A)$.
		
		It remains to show that $K$ is a $G$-map. Let $g\in G$, $t\in I$ and $p\in C(X,Y)$. Then, for every $x\in X$,
		
		\begin{center}
			$K_t(g\cdot p)(x)=F_t\big((g\cdot p)(x)\big)=K_t(p)(g^{-1} x)=\big(g\cdot K_t(p)\big)(x)$.
		\end{center}
		This shows that $K_t(g\cdot p)=g\cdot K_t(p)$, and hence $K$ is equivariant. Therefore, $C(X,A)$ is $G$-homotopy dense in $C(X,Y)$.
	\end{proof}
	
	\begin{lemma}\label{ejemplouniformretract}
		\rm{Let $X$ be a compact $G$-space and $(Y,d)$ a metric space. If $B$ is a uniform retract (respectively, a uniform neighborhood retract) of $Y$, then $C(X,B)$ is a uniform $G$-retract (respectively, a uniform neighborhood $G$-retract) of $(C(X,Y),d^*)$.}
	\end{lemma}
	
	\begin{proof}
		
	Suppose that $B$ is a uniform neighborhood retract of $Y$.
	
	Let $U$ be a uniform neighborhood of $B$ in $Y$, and let $r: U\rightarrow B$ be a retraction that is uniformly continuous at $B$. Notice that $C(X,U)$ is a uniform $G$-neighborhood of $C(X,B)$ in $C(X,Y)$.
	
	Define a continuous map $R:C(X,U)\rightarrow C(X,B)$ by
	
	\begin{center}
		$R(p)=r\circ p,\quad$ $p\in C(X,U)$.
	\end{center}
	
	For each $p\in C(X,U)$, $g\in G$ and $x\in X$, we have
	\begin{center}
		$R(g\cdot p)(x)=r\big(p(g^{-1} x)\big)=(r\circ p)(g^{-1} x)=\big(g\cdot R(p)\big)(x)$,
	\end{center}
	showing that $R(g\cdot p)=g\cdot R(p)$. Hence, $R$ is a $G$-retraction.
	
	We now verify that $R$ is uniformly continuous at $C(X,B)$. Given $\varepsilon > 0$, choose $\delta > 0$ such that $d(r(u),b)<\frac{\varepsilon}{2}$ for all $u\in U$, $b\in B$ satisfying $d(u,b)<\delta$.
	
	Let $p\in C(X,U)$ and $q\in C(X,B)$ with $d^*(p,q)<\delta$. For each $x\in X$, we have $d(p(x),q(x))<\delta$, which implies $d(r(p(x)),q(x))<\frac{\varepsilon}{2}$. Consequently, $d^*(R(p),q)\leq \frac{\varepsilon}{2}<\varepsilon$. Therefore, $R$ is uniformly continuous at $C(X,B)$, as required.
		
	\end{proof}
	
	\noindent\textit{Proof of Theorem \ref{functionC(X,Y)GUANR}.}
		Suppose that $Y$ is a UANR, and let $(\widetilde{Y},d)$ be the metric completion of $Y$. By \cite[Theorem 2]{KS1}, $\widetilde{Y}$ is itself a UANR, and $Y$ is homotopy dense in $\widetilde{Y}$.
		
		The Arens-Eells Embedding Theorem (\cite[Theorem 6.2.1]{KS2}) guarantees the existence of a Banach space $L$ containing $\widetilde{Y}$ as a closed subset. Since $\widetilde{Y}$ is a uniform neighborhood retract of $L$, it follows from Lemma \ref{ejemplouniformretract} that the $G$-space $C(X,\widetilde{Y})$ is a  uniform neighborhood $G$-retract of $C(X,L)$.

		Now, Lemma \ref{UAE grupo compacto no Lie} shows that the normed linear $G$-space $C(X,L)$ is a $G$-UAR. Consequently, Lemma \ref{retractuniform} implies that $C(X,\widetilde{Y})$ is a $G$-UANR. In addition, by Lemma \ref{functionhomotopydense}, $C(X,Y)$ is $G$-homotopy dense in $C(X,\widetilde{Y})$. Therefore, by Lemma \ref{G-UANRpreserveHD}, the function space $C(X,Y)$ is a $G$-UANR.
		
		Next, assume that $Y$ is a UAR. Then $C(X,Y)$ is a $G$-UANR, and by \cite[Theorem 8]{A3}, $C(X,Y)$ is also a $G$-AR. Hence, by Lemma \ref{G-AR+G-UANR=G-UAR}, it follows that $C(X,Y)$ is a $G$-UAR.
		
		Conversely, suppose that $C(X,Y)$ is a $G$-UAR (respectively, a $G$-UANR). Since $C(X,Y)$ is a UAR (respectively, a UANR), and $Y$ is isometric to a uniform retract of $C(X,Y)$, it follows that $Y$ is a UAR (respectively, a UANR). This finishes the proof.
		
	\qed
	

	\section{Applications to Lawson \texorpdfstring{$G$-}-semilattices}\label{sec4}
	
	A topological semilattice is a topological space $S$ equipped with a continuous operation $\cdot: S\times S\rightarrow S$ that is idempotent, commutative, and associative (i.e., $x\cdot x=x$, $x\cdot y=y\cdot x$ and $x\cdot (y\cdot z)=(x\cdot y)\cdot z$). Moreover, $S$ is called a Lawson semilattice if it has a basis of subsemilattices, that is, there exists a basis $\mathcal{B}$ for the topology of $S$ such that each $B\in \mathcal{B}$ is a subsemilattice of $S$.  
	
	A topological $G$-semilattice is a topological semilattice $S$ endowed with a continuous action of $G$ such that the multiplication map is equivariant; that is, for all $x,y \in S$ and $g \in G$, the following condition holds:
	\begin{center}
		$g\cdot(xy)=(g\cdot x)(g\cdot y)$.
	\end{center}
	
	If, in addition, $S$ is a Lawson semilattice, then $S$ is called a Lawson $G$-semilattice. An invariant subsemilattice of $S$ is said to be a $G$-subsemilattice.
			
	\begin{proposition}\label{Cone(X)semilattice}
		\rm{Let $S$ be a $G$-semilattice, and let $\mathrm{Cone}(S)$ be the cone of $S$ equipped with the $G$-action $g\cdot sx=s(gx)$. Consider the map 
		$$\cdot: \mathrm{Cone}(S)\times \mathrm{Cone}(S)\rightarrow \mathrm{Cone}(S)$$
		 defined by 
		 $$sx\cdot ty=\mbox{min}\{s,t\}(xy).$$ 
		 Then the following conditions are satisfied:
			
			\begin{enumerate}
				\item [(i)] $(\mathrm{Cone}(S),\cdot)$ is a $G$-semilattice.
				\item [(ii)] If $S$ is a Lawson $G$-semilattice, then so is $\mathrm{Cone}(S)$.
		\end{enumerate}}
	\end{proposition}
	
	\begin{proof}
		
		(i) We first prove that the multiplication is continuous. To this end, let $\eta: I\times S\rightarrow \mathrm{Cone}(S)$ denote the quotient map. Take $sx, ty\in \mathrm{Cone}(S)$, and let $W$ be an open neighborhood of $sx\cdot ty=\text{min}\{s,t\}(xy)$ in $\mathrm{Cone}(S)$. We consider two cases:
		\begin{itemize}
			\item \textbf{Case 1}: $sx\neq \theta\neq ty$. In this case, we have $\text{min}\{s,t\}(xy)\neq \theta$, and thus, $\eta^{-1}(W)$ is an open neighborhood of $(\text{min}\{s,t\}, xy)$ in $I\times S$. Therefore, there exist $\delta>0$ and neighborhoods $V_x$, $V_y$ of $x$ and $y$, respectively, such that $0<s-\delta<s+\delta<1$, $0<t-\delta<t+\delta<1$ and $\text{min}\big((s-\delta,s+\delta)\times (t-\delta,t+\delta)\big)\times V_xV_y\subset \eta^{-1}(W)$. 
			
			Now, set
			
			\begin{center}
				$U_x=\eta\big((s-\delta,s+\delta)\times V_x\big)\quad$ and $\quad U_y=\eta\big((t-\delta,t+\delta)\times V_y\big)$.
			\end{center}
			
			Then $U_x$ and $U_y$ are neighborhoods of $sx$ and $ty$ in $\mathrm{Cone}(S)$, respectively, and satisfy $U_x\cdot U_y\subset W$. Hence, the multiplication is continuous at $(sx,ty)$.
			
			\item \textbf{Case 2}: $sx=\theta$ or $ty=\theta$. Without loss of generality, suppose that $sx=\theta$. By definition of the multiplication, we have $sx\cdot ty=\theta$. Hence, there exists $\varepsilon > 0$ such that $[0,\varepsilon)\times S\subset \eta^{-1}(W)$. Notice that $\eta([0,\varepsilon)\times S)$ and $\mathrm{Cone}(S)$ are open neighborhoods of $sx$ and $ty$, respectively, and moreover $\eta([0,\varepsilon)\times S)\cdot \mathrm{Cone}(S)\subset W$. Thus, the multiplication is continuous at $(sx,ty)$.
		\end{itemize}

		On the other hand, the multiplication is associative. Indeed, take $rx,sy,tz\in \mathrm{Cone}(S)$, and let us prove that $rx\cdot(sy\cdot tz)=(rx\cdot sy)\cdot tz$. If $\theta \in \{rx, sy, tz\}$, then both sides equal $\theta$. Otherwise,
		\begin{center}
			$rx\cdot(sy\cdot tz)=\mbox{min}\{r,\mbox{min}\{s,t\}\}(x(yz))=\mbox{min}\{\mbox{min}\{r,s\},t\}((xy)z)=(rx\cdot sy)\cdot tz$.  
		\end{center}
		
		Finally, we verify equivariance. Given $g\in G$ and $rx,sy\in \mathrm{Cone}(S)$, we have
		
		\begin{center}
			$g\cdot (rx\cdot sy)=g\cdot (\mbox{min}\{r,s\}(xy))=\mbox{min}\{r,s\}(g\cdot (xy))=\mbox{min}\{r,s\}((gx)\cdot (gy))=(g\cdot rx)(g\cdot sy)$.
		\end{center}
		Therefore, $(\mathrm{Cone}(S),\cdot)$ is a topological $G$-semilattice.   
		
		(ii) It suffices to prove that any point of $\mathrm{Cone}(S)$ admits a neighborhood basis consisting of subsemilattices of $\mathrm{Cone}(S)$. Let $sx\in \mathrm{Cone}(S)$ and let $V$ be an open neighborhood of $sx$.
		
		\begin{itemize}
			\item \textbf{Case 1}: $sx\neq \theta$. In this case, $s > 0$, and hence $\eta^{-1}(V)$ is an open neighborhood of $(s,x)$ in $I\times S$. Since $S$ is a Lawson semilattice, there exists a neighborhood $U$ of $x$ in $S$ such that $U$ is a semilattice and $(s-\varepsilon,s+\varepsilon)\times U\subset \eta^{-1}(V)$, for some $\varepsilon>0$ with  $0<s-\varepsilon<s+\varepsilon<1$. Then $\eta\big((s-\varepsilon,s+\varepsilon)\times U\big)$ is a neighborhood of $sx$, contained in $V$, which is itself a subsemilattice of $\mathrm{Cone}(S)$.
			
			\item \textbf{Case 2}: $sx=\theta$. In this case, there exists $\delta>0$ such that $[0,\delta)\times S\subset \eta^{-1}(V)$. Thus, $U=\eta([0,\delta)\times S)$ is the desired neighborhood of $sx$.
		\end{itemize}
	\end{proof}
		
	A Lawson metric $G$-semilattice is a triple $(S,d,\cdot)$, where $(S,\cdot)$ is a Lawson $G$-semilattice and $d$ is an admissible invariant metric on $S$ satisfying: $d(a\cdot b,a'\cdot b')\leq\mbox{max}\{d(a,a'),d(b,b')\}$ for all $a,a',b,b'\in S$. When $G=\{e\}$, this definition reduces to the notion of a Lawson metric semilattice (see, e.g., \cite{Banakh1,KSakaiY}).
	
	Let $G$ be a compact group and let $(S,d,\cdot)$ be a Lawson metric semilattice equipped with an action of $G$. We prove that there exists an admissible metric $\rho$ on $S$ such that $(S,\rho,\cdot)$ is a Lawson metric $G$-semilattice. The function $\rho:S\times S\rightarrow \mathbb{R}$ defined by 
	\begin{center}
		$\rho(x,y)=\text{sup}\, \{d(gx,gy)\mid g\in G\},\quad$  $x,y\in S$
	\end{center}
	is an invariant admissible metric on $S$. Thus, it remains to check that 
	
	\begin{center}
		$\rho(a b,a' b')\leq\mbox{max}\{\rho(a,a'),\rho(b,b')\}\quad$ for all $a,a',b,b'\in S$.
	\end{center}
	
	For each $g\in G$ we have
	\begin{center}
		$d\big(g(ab),g(a'b')\big)=d\big((ga)(gb),(ga')(gb')\big)\leq \mbox{max}\{d(ga,ga'),d(gb,gb')\}\leq \mbox{max}\{\rho(a,a'),\rho(b,b')\}$,
	\end{center}
	and therefore, 
	\begin{center}
		$\rho(a b,a' b')=\text{sup}\,\{d\big(g(ab),g(a'b')\big)\mid g\in G\}\leq\mbox{max}\{\rho(a,a'),\rho(b,b')\}$ 
	\end{center}
	which shows that $(S,\rho,\cdot)$ is a Lawson metric $G$-semilattice.
	
	\begin{proposition}\label{FuncspacesLawonmetric}
		\rm{Let $S$ be a topological semilattice and $X$ a locally compact $G$-space. Consider the binary operation 
		$$\cdot: C(X,S)\times C(X,S)\rightarrow C(X,S)$$
		 defined by
			
			\begin{center}
				$(p\cdot q)(x)=p(x)q(x),\quad$ $p,q\in C(X,S),\ \ $ $x\in X$. 
			\end{center}
			
			Then the following conditions are satisfied:
			
			\begin{enumerate}
				\item [(i)] $(C(X,S),\cdot)$ is a topological $G$-semilattice.
				\item [(ii)] If $S$ is a Lawson semilattice, then $C(X,S)$ is a Lawson $G$-semilattice.
				\item [(iii)] If $(S,d,\cdot)$ is a Lawson metric semilattice and $X$ is compact, then $(C(X,S),d^*,\cdot)$ is a Lawson metric $G$-semilattice.
		\end{enumerate}}
	\end{proposition}
	
	\begin{proof}
		
		(i) First, let us show that the multiplication is continuous. To this end, it suffices to prove that the set 
		\begin{center}
			$Z=\{(p,q)\in C(X,S)\times C(X,S)\mid p\cdot q\in [K,W] \}$
		\end{center}
		is open for any compact subset $K\subset X$ and any open subset $W\subset S$. 
		
		Let $(p,q)\in Z$. Then for each $k\in K$ there exist open subsets $U_k$ and $V_k$ of $S$ such that $\big(p(k),q(k)\big)\in U_k\times V_k$ and $U_k V_k\subset W$. 
		
		Now define a continuous map $p\times q:X\rightarrow S\times S$ by $(p\times q)(x)=\big(p(x),q(x)\big)$. Since $X$ is locally compact, there is a neighborhood $O_k$ of $k$ in $X$ such that $\overline{O_k}$ is compact and $(p\times q)(\overline{O_k})\subset U_k\times V_k$. 
		
		By the compactness of $K$, there exist $n\in\mathbb{N}$ and points $k_1,\dots,k_n\in K$ such that $K\subset \bigcup\limits_{i=1}^nO_{k_i}$. Observe that
		
		\begin{center}
			$p\in U= \bigcap\limits_{i=1}^n[\overline{O_{k_i}},U_{k_i}]\quad$ and $\quad q\in  V= \bigcap\limits_{i=1}^n[\overline{O_{k_i}},V_{k_i}]$.
		\end{center}
		
		Moreover, $U\times V\subset Z$. Indeed, given $(m,n)\in U\times V$ and $k\in K$, there exists $j\in\{1,\dots,n\}$ such that $k\in O_{k_j}$. Consequently,
		
		\begin{center}
			$(m\cdot n)(k)=m(k) n(k)\in U_{k_j} V_{k_j}\subset W$,
		\end{center}
		which shows that $m\cdot n\in [K,W]$. Therefore, $Z$ is open and $(C(X,S),\cdot)$ is a topological semilattice.
		
		On the other hand, let $p,q\in C(X,S)$ and $g\in G$. Then, for each $x\in X$, we have:
		
		\begin{center}
			$\big(g(p\cdot q)\big)(x)=(p\cdot q)(g^{-1} x)=p(g^{-1} x)q(g^{-1} x)=(gp\cdot gq)(x)$.
		\end{center}
		
		This implies that $g(p\cdot q)=(gp)\cdot (gq)$, proving that the operation is equivariant. Hence, $C(X,S)$ is a topological $G$-semilattice.
		
		(ii) Let $\mathcal{A}$ be an open basis for the topology of $S$ consisting of subsemilattices.  According to \cite[Chapter XXII, Proposition 5.1]{JD2}, a subbasis for the compact-open topology is given by 
		$$
			\mathcal{B}=\{[K,U]\mid K\subset X\ \mbox{compact\ and}\ U\in \mathcal{A}\}.
		$$
		
		Note that $[K,U]$ is a subsemilattice for each compact subset $K\subset X$ and each $U\in \mathcal{A}$. Thus, $C(X,S)$ is a Lawson $G$-semilattice.

		(iii) Let $p,q,m,n \in C(X,S)$. We can suppose that $d^*(q,n)\leq d^*(p,m)$. Then, for each $x\in X$, we have
		
		\begin{center}
			$d\big(p(x) q(x), m(x) n(x)\big)\leq \max\{d\big(p(x),m(x)\big), d\big(q(x), n(x)\big)\}\leq d^*(p,m)$.
		\end{center}
		
		This implies that
		\begin{center}
			$d^*(p\cdot q, m\cdot n)=\sup\{d\big(p(x)q(x), m(x) n(x)\big)\mid x\in X\}\leq d^*(p,m)$.
		\end{center}
		Therefore, $d^*(p\cdot q, m\cdot n)\leq \max\{d^*(p,m),d^*(q,n)\}$ and $(C(X,S),d^*,\cdot)$ is a Lawson metric $G$-semilattice. This completes the proof.
	\end{proof}

	Another interesting example of a Lawson metric $G$-semilattice  is the hyperspace $$2^X=\{A\subset X\mid A\ \rm{is\ compact\ and\ nonempty}\}$$ of a metric $G$-space $(X,d).$ Here, $2^X$ is equipped with the action of $G$ given by $g\cdot A=\{ga\mid a\in A\}$ for each $g\in G$ and $A\in 2^X$. Furthermore, it is endowed with the Hausdorff metric $d_H:2^X\times 2^X\rightarrow [0,\infty)$ defined by
	\begin{center}
		$d_H(A,B)=\max \big\{\sup \{d(a,B)\mid a\in A\},\ \sup\{d(b,A)\mid b\in B\}\big\}$
	\end{center} 
	for each $A, B\in 2^X$.
	
	Notice that 
	$$\cup: 2^X\times 2^X\rightarrow 2^X,\quad (A,B)\mapsto A\cup B$$
	 is a $G$-map and $(2^X,d_H,\cup)$ is a Lawson metric $G$-semilattice. 
	 
	 Moreover, since $F(X)=\{A\in 2^X\mid A\ \mbox{is finite}\}$ is an invariant subsemilattice of $2^X$,  it follows that $\big(F(X),d_H,\cup\big)$ is also a Lawson metric $G$-semilattice.
	
	Recall that a metric space $(X,d)$ is called \textit{uniformly locally path-connected} (see \cite{Banakh1}, \cite{KS1} and \cite{MYaguchi}), if for each $\varepsilon>0$, there exists $\delta>0$ such that for any $a, b\in X$ with $d(a,b)<\delta$, there exists a path $f: I\rightarrow X$ such that $\mbox{diam}\, f(I)<\varepsilon$, $f(0)=a$ and $f(1)=b$.
	
	In Theorem \ref{SakaiGfinito}, we extend Theorem \ref{Sakai2005Teo3.4} to the setting of $G$-spaces in two cases: either $G$ is finite, or both $G$ and $S$ are compact. For the proof, we require the following result.
	
	\begin{lemma}[\cite{A10}]\label{hyperauto}
		\rm{If $(X,d)$ is a uniformly locally path-connected metric $G$-space, then $2^X$ is a $G$-UANR. Moreover, if $X$ is also connected, then $2^X$ is a $G$-UAR.}
	\end{lemma}

	\begin{theorem}\label{SakaiGfinito}
		\rm{Let $G$ be a finite group and $(S,d,\cdot)$ a Lawson metric $G$-semilattice. Then the following statements are equivalent:
			\begin{enumerate}
				\item [(i)] $S$ is a $G$-UANR,
				\item [(ii)] $S$ is a weak $G$-UANR,
				\item [(iii)] $S$ is $G$-ULC,
				\item [(iv)] $S$ is uniformly locally path-connected.
			\end{enumerate}
			
			Moreover, the equivalence also holds if both $G$ and $S$ are compact.}
		
	\end{theorem}
	
	\begin{proof}
		By Theorem \ref{lemautil}, it suffices to verify the implication (iv) $\Rightarrow$ (i). To this end, suppose that $S$ is uniformly locally path-connected. By Theorem \ref{Sakai2005Teo3.4}, $S$ is also a UANR. We shall prove that $S$ is a $G$-UANR.
		
		Let $A$ be a closed $G$-subset of a metric $G$-space $(X,\rho)$, and let $j:A\rightarrow S$ be a uniformly continuous $G$-map. Since $S$ is a UANR, there exist a uniform neighborhood $U$ of $A$ in $X$ and a continuous map $k: U\rightarrow S$ that is uniformly continuous at $A$. We may assume that $U$ is invariant.
		
		Now consider the $G$-space $(F(S),d_H)$ and the map $K:U\rightarrow F(S)$ defined by $K(u)=\{g^{-1} k(gu)\mid g\in G\}$ for each $u\in U$.
		
		\smallskip
		
		\noindent\textit{Claim 1.} $K$ is a $G$-map and is uniformly continuous at $A$.
		
		First we prove that  $K$ is continuous at every point $u_0\in U$.	Let $\varepsilon>0$. For each $g\in G$, there exists $\delta_g>0$ such that $d(k(gu),k(gu_0))<\varepsilon$ whenever $u \in U$ satisfies $\rho(gu,gu_0)<\delta_g$. Choose $0<\delta<\text{min}\{\delta_g\mid g\in G\}$. Then, for any $u\in U$ with $\rho(u,u_0)<\delta$, we have $\rho(gu, gu_0) < \delta_g$ for all $g \in G$. Hence, 
		\begin{center}
			$d(g^{-1} k(gu),g^{-1} k(gu_0))=d(k(gu),k(gu_0))<\varepsilon$,
		\end{center}
		which shows that $d_H(K(u),K(u_0))<\varepsilon$. Therefore, $K$ is continuous. Moreover, it follows directly that $K$ is equivariant.
		
		To show that $K$ is uniformly continuous at $A$, let $\varepsilon>0$ and choose $\delta>0$ such that $d(k(u),k(a))<\varepsilon$ for each $u\in U$ and $a\in A$ with $\rho(u,a)<\delta$. Observe that, for each $g\in G$,
		\begin{center}
			$d(g^{-1} k(gu),k(a))=d(g^{-1} k(gu),g^{-1} k(ga))=d(k(gu),k(ga))<\varepsilon$.
		\end{center}
		
		It follows that 
		$$d_H(K(u),K(a))=d_H(\{g^{-1} k(gu)\mid g\in G\},\{k(a)\})<\varepsilon.$$
		 Thus, $K$ is uniformly continuous at $A$ and  this completes the proof of Claim 1.
		
		\smallskip
		
		On the other hand, consider the map $r: F(S)\rightarrow S$ defined by $r(\{s_1,\dots,s_n\})=s_1\cdots s_n$. Notice that $r$ is a continuous retraction. In fact, it is clear that $r$ is equivariant.
		
		\smallskip
		
		\noindent\textit{Claim 2.} $r$ is a  uniformly continuous      $G$-retraction.   
		
		Indeed, let $\{a_1,\dots,a_m\}, \{b_1,\dots,b_n\}\in F(S)$. We may assume that $m=n$. Therefore
		\begin{align*}
			d\big(r(\{a_1,\dots,a_m\}),r(\{b_1,\dots,b_m\})\big)& =d(a_1\cdots a_m,b_1\cdots b_m)\\
			& \leq \text{max}\, \{d(a_i,b_i)\mid i\in \{1,\dots,m\}\}\\
			& \leq d_H(\{a_1,\dots,a_m\},\{b_1,\dots,b_m\}).
		\end{align*} 
		
		Then $r$ is uniformly continuous and this proves the Claim.
		
		Finally, consider the $G$-map $J=r\circ K: U\rightarrow S$. Observe that $J$ is uniformly continuous at $A$. Moreover, for each $a\in A$, we have
		
		\begin{center}
			$J(a)=r(\{g^{-1} k(ga)\mid g\in G\})=r(\{j(a)\})=j(a)$
		\end{center}
		which shows that $J$ is an extension of $j$. This establishes the implication  (iv) $\Rightarrow$ (i).
		
		Next, we consider the case where both $G$ and $S$ are compact. We prove the implication  (iv) $\Rightarrow$ (i). 
		
		Since $S$ is uniformly locally path-connected, by Lemma \ref{hyperauto} we get that $2^S$ is a $G$-UANR. Therefore, it suffices to prove that $S$ is a uniform neighborhood $G$-retract of $2^S$.
		
		By the compactness of $S$, it follows from \cite{MMC} that the map $r: 2^S\rightarrow S$ defined by $r(A)=\mbox{inf}(A)$  is a continuous retraction. Moreover, the map $r$ is uniformly continuous at $S$.  Additionally, since the multiplication in $S$ is equivariant, for each $A\in 2^S$ and $g\in G$ we have $r(gA)=\mbox{inf}(gA)=g\, \mbox{inf}(A)$, that is, $r$ is a $G$-map. Hence, $S$ is a uniform $G$-retract of $2^S$. This completes the proof.

	\end{proof}
	
	When $G$ is a compact group and $S$ is an arbitrary Lawson metric $G$-semilattice, the previous result has not yet been proved. However, there exist Lawson metric $G$-semilattices for which the result does hold. One such example is given by the hyperspace $2^X$ of a metric $G$-space $X$, where $G$ is a compact group (see \cite{A10}). In Proposition \ref{EquivaC(X,S)equivariant}, we present another example as a consequence of Theorem \ref{functionC(X,Y)GUANR}.

	\begin{proposition} \label{EquivaC(X,S)equivariant}
		\rm{Let $G$ be a compact group, $X$ a compact $G$-space and $(S,d,\cdot)$ a Lawson metric semilattice. Then the following statements are equivalent for the $G$-space $(C(X,S),d^*)$:
			
			\begin{enumerate}
				\item [(i)] $C(X,S)$ is a $G$-UANR,
				\item [(ii)] $C(X,S)$ is a weak $G$-UANR,
				\item [(iii)] $C(X,S)$ is $G$-ULC,
				\item [(iv)] $C(X,S)$ is uniformly locally path-connected.
		\end{enumerate}}
		
	\end{proposition}
	
	\begin{proof}
		It is enough to verify the implication (iv) $\Rightarrow$ (i). 
		
		Suppose that $C(X,S)$ is uniformly locally path-connected. We claim that $S$ is also uniformly locally path-connected. 
		
		Indeed, let $\varepsilon>0$ be given. Then there exists $\delta>0$ such that for any $p,q\in C(X,S)$ with $d^*(p,q)<\delta$, there is a path $\nu: \I\rightarrow C(X,S)$ with $\nu(0)=p$, $\nu(1)=q$ and $\text{diam}\, \nu(\I)<\frac{\varepsilon}{2}$. 
		
		Now, take $a,b\in S$ with $d(a,b)<\delta$, and define the constant maps $\iota_a,\iota_b:X\rightarrow S$ by $\iota_a(x)=a$ and $\iota_b(x)=b$ for all $x\in X$. Then there exists a path $\nu:\I\rightarrow C(X,S)$ from $\iota_a$ to $\iota_b$ so that $\text{diam}\, \nu(\I)<\frac{\varepsilon}{2}$. 
		
		Fix $x_0\in X$ and consider the map $\tau: \I\rightarrow S$ defined by $\tau(t)=\nu(t)(x_0)$. Then $\tau$ is a path from $a$ to $b$ such that $\text{diam}\, \tau(\I)<\varepsilon$. This shows that $S$ is uniformly locally path-connected. Hence, by Theorem \ref{Sakai2005Teo3.4}, $S$ is a UANR, and by Theorem \ref{functionC(X,Y)GUANR}, $C(X,S)$ is a $G$-UANR. This completes the proof.
	\end{proof}

	\bibliographystyle{amsplain}

\end{document}